\documentclass[a4paper,11pt]{article}
\usepackage[T1]{fontenc}
\usepackage{rotating}
\usepackage[title]{appendix}
\usepackage{slashbox}
\usepackage{upgreek}
\usepackage{amsmath} 
\usepackage{amsthm}
\usepackage{amsfonts}
\usepackage{mathtools}
\usepackage{amssymb}
\usepackage{enumitem}
\usepackage{gensymb}

\def\x{\mathbf{x}}

\DeclareMathOperator{\Aut}{Aut}

\usepackage{graphicx}
\usepackage{url}
\usepackage[hyperfootnotes=true,colorlinks=true]{hyperref}
\hypersetup{
    colorlinks=true,
    linkcolor=blue,
    citecolor=magenta,      
    urlcolor=black
}

\usepackage{tikz}
\usetikzlibrary{calc}
\usetikzlibrary{backgrounds}
\usetikzlibrary{intersections}

\usepackage{authblk}
\usepackage{subcaption}
 
\usepackage[text={170mm,254mm}]{geometry}

\theoremstyle{plain}
\newtheorem{theorem}{Theorem}
\newtheorem{lemma}[theorem]{Lemma}

\newtheorem{proposition}[theorem]{Proposition}

\theoremstyle{definition}

\newtheorem{problem}[theorem]{Problem}

\usepackage{ifxetex,ifluatex}
\if\ifxetex T\else\ifluatex T\else F\fi\fi T%
  \usepackage{fontspec}
\else
  \usepackage[T1]{fontenc}
  \usepackage[utf8]{inputenc}
  \usepackage{lmodern}
\fi

\usepackage{hyperref}

\title{Nut graphs with a given automorphism group}

\author[1,2,3]{Nino Ba{\v s}i{\'c}}
\author[4]{Patrick~W.~Fowler}
 
\affil[1]{FAMNIT, University of Primorska, Koper, Slovenia}
\affil[2]{IAM, University of Primorska, Koper, Slovenia}
\affil[3]{Institute of Mathematics, Physics and Mechanics, Ljubljana, Slovenia}
\affil[4]{Department of Chemistry, University of Sheffield, Sheffield S3 7HF, UK}
\begin{document}

\maketitle

\begin{abstract}
A \emph{nut graph} is a simple graph of order 2 or more for which the adjacency matrix has a 
single zero eigenvalue such that all non-zero kernel eigenvectors have no zero entry (i.e.\ are full). 
It is shown by construction
that every finite group can be represented as the group of automorphisms of
infinitely many nut graphs. It is further shown that such nut graphs exist even within the class of regular graphs;
the cases where the degree is $8, 12, 16, 20$ or $24$ are realised explicitly.

\vspace{\baselineskip}

\vspace{\baselineskip}
\noindent
\textbf{Math.\ Subj.\ Class.\ (2020):} 
05C25, % Graphs and abstract algebra (groups, rings, fields, etc.)
05C50. % Graphs and linear algebra (matrices, eigenvalues, etc.)
\end{abstract}

\section{Introduction}

A problem posed in K{\H o}nig's 1936 book on Graph Theory \cite[p.~5]{Konig1936} asks
when a given abstract group can be represented as the group of
automorphisms of a (finite) graph $G$, and when this is the case, how
the graph can be constructed\footnote{``Wann kann eine gegebene abstrakte Gruppe als die
Gruppe eines Graphen aufgefa\ss{}t werden und -- ist des der Fall --
wie kann die entsprechende Graph konstruoert werden?''}.
In response, Frucht first solved the problem in its original form~\cite{Frucht1939}.
Later, he showed that solution is still possible under the extra requirement that $G$ is a cubic graph~\cite{Frucht1949}. In both cases he gave an explicit construction.
Sabidussi~\cite{Sabidussi1957} refined the question and proved that every group can be represented by a graph 
with additional properties such as: prescribed chromatic number, prescribed vertex-connectivity, or regularity with prescribed degree.
An early survey paper by Babai~\cite{Babai1981} reviews 
this research direction and defines the term \emph{f-universal}: a class of graphs $\mathcal{C}$ is f-universal if 
for every finite group $\mathfrak{G}$ there exists a graph $G \in \mathcal{C}$ such that $\Aut(G) \cong \mathfrak{G}$.
Not all famous graph classes are f-universal; for example, Babai has shown that there are infinitely many finite groups which
cannot be realised by a \emph{planar} graph~\cite{Babai1972,Babai1975}.
K{\H o}nig's question has also been extended from graphs to other combinatorial objects, 
such as tournaments \cite{Moon1964}, Steiner triple and quadruple systems \cite{Mendelsohn1978}, and cycle systems \cite{Grannell2013,Lovegrove2014}.
Here, we consider K{\H o}nig's original question, but for \emph{nut graphs}.

A nut graph is a singular simple graph with \emph{nullity} 1, where the non-trivial kernel eigenvector 
has only non-zero entries. Nut graphs occur in several chemical applications \cite{nuttybook}: they are
connected, leafless and non-bipartite~\cite{ScirihaGutman-NutExt}. Catalogues have been constructed \cite{nutgen-site, CoolFowlGoed-2017,hog,HoG2}: nut graphs may be
regular, vertex-transitive \cite{Jan2020, basic2021regular} (including GRRs \cite{WatkinsGRR1,WatkinsGRR2} and non-Cayley graphs), but 
are not edge-transitive~\cite{basic2023vertex}; they may be chemical graphs \cite{Basic2020}, 
including some cubic polyhedra \cite{feuerpaper} and, in particular, fullerenes~\cite{ScFo2008}. Recently, a comprehensive theory of circulant nut graphs has been developed \cite{DamnjanStevan2022,Damnjanovic2022b,Damnjanovic2022c,Damnjanovic2022a}. The study of polycirculant nut graphs \cite{damnjanovic2023tricirc} has also been initiated.

Here, we prove a Sabidussi-type result for the class of nut graphs: we show that the graph $G$ that realises a given automorphism group
can be chosen to satisfy the requirements of the nut-graph definition; 
hence, nut graphs are f-universal (in the sense of \cite{Babai1981}). We prove that:

\begin{theorem}
\label{thm:1}
For every finite group $\mathfrak{G}$ there exist infinitely many finite nut graphs $G$, such that $\Aut(G) \cong \mathfrak{G}$.
\end{theorem}

\noindent
Furthermore, we can require that the graph $G$ is also regular.

\begin{theorem}
\label{thm:2}
For every finite group $\mathfrak{G}$ and $d \in \{8, 12, 16, 20, 24\}$ there exist infinitely many finite $d$-regular nut graphs $G$, such that $\Aut(G) \cong \mathfrak{G}$.
\end{theorem}

\section{Preliminaries}

All graphs considered in this paper are finite, simple and connected.
The adjacency matrix of graph $G$ is $\mathbf{A}(G)$
and the dimension of the nullspace of $\mathbf{A}(G)$
is the \emph{nullity},  $\eta(G)$.
An \emph{automorphism} $\alpha$ of a graph $G$ is a permutation $\alpha\colon V(G) \to V(G)$ of the vertices
of $G$ that maps edges to edges and non-edges to non-edges.
The set of all automorphisms of a graph $G$ forms a group,
the (\emph{full}) \emph{automorphism group} of $G$, denoted by $\Aut(G)$.
The image of a vertex $v \in V(G)$ under automorphism $\alpha$ will be denoted $v^\alpha$.
For other standard definitions we refer the reader to one of 
the many comprehensive treatments of the theory of graph spectra (e.g.~\cite{Cvetkovic1997,Cvetkovic2010,Chung1997,Cvetkovic1995,haemers}).
and algebraic graph theory (e.g.~\cite{Godsil2001,Dobson2022,Biggs}). 

\emph{Nut graphs} \cite{ScirihaGutman-NutExt} are graphs that have a one-dimensional nullspace (i.e., $\eta(G) = 1$), 
where the non-trivial kernel eigenvector $\x = [x_1\ \ldots\ x_n]^\intercal \in \ker \mathbf{A}(G)$ is full (i.e., $|x_i| > 0$ for all $i = 1, \ldots, n$). 
As the defining paper considered the isolated vertex 
to be a trivial case~\cite{ScirihaGutman-NutExt}, \emph{non-trivial}
nut graphs have seven or more vertices. 
If $G$ is a \emph{regular} nut graph, then $\delta(G) = d(G) = \Delta(G) \geq 3$.
Note that there are no nut graphs with $\Delta(G) = 2$, as no cycle has nullity $1$.

In what follows, it will be useful to have constructions that
are guaranteed to produce a nut graph, when applied to a starting graph of specified type. 
For example, let $G$ be a nut graph and $e \in E(G)$ an arbitrary edge. Then the graph obtained from $G$
by subdividing the edge $e$ four times is again a nut graph; this is the \emph{subdivision construction}~\cite{ScirihaGutman-NutExt}.
Two further constructions that will prove useful in what follows are now described.

The first is the coalescence construction:
Let $G_1$ and $G_2$ be graphs and let $v_1 \in V(G_1)$ and $v_2 \in V(G_2)$. The \emph{coalescence of $(G_1, v_1)$ and $(G_2, v_2)$}, which we denote here as $(G_1, v_1) \odot (G_2, v_2)$,
is the graph obtained from the disjoint union of $G_1$ and $G_2$ by identifying root vertices $v_1$ and $v_2$.
Sciriha obtained the following result~\cite[Corollary 21]{Sciriha2008}.

\begin{lemma}[\cite{Sciriha2008}]
\label{lem:coalescence}
Let $G_1$ and $G_2$ be nut graphs. Then the coalescence $(G_1, v_1) \odot (G_2, v_2)$ is a nut graph.
\end{lemma}

\noindent
The coalescence construction must be provided with an initial collection of nut graphs.
The second construction is different in that it produces a nut graph from any $(2t)$-regular graph.

\begin{proposition}[\cite{basic2023vertex}]
\label{prop:multiplier3}
Let $G$ be a connected $(2t)$-regular graph, where $t \geq 1$. Let $\mathcal{M}_3(G)$ be the graph obtained from $G$ by
fusing a bouquet of $t$ triangles to every vertex of $G$. Then $\mathcal{M}_3(G)$ is a nut graph.
\end{proposition}

\noindent
The construction $\mathcal{M}_3(G)$ is called the \emph{triangle-multiplier construction}~\cite{basic2023vertex}.
The choice of name is justified by the fact that $|V(\mathcal{M}_3(G))| = (2t+1)|V(G)|$. Its effect on the automorphism
group is described by the following proposition.

\begin{proposition}[\cite{basic2023vertex}]
\label{prop:multiConstrSym}
Let $G$ be a connected $(2t)$-regular graph, where $t \geq 1$. 
Then $\Aut(G) \leq \Aut(\mathcal{M}_3(G))$  and $|\Aut(\mathcal{M}_3(G))| = (2^t t!)^{|V(G)|} |\Aut(G)|$.
\end{proposition}

\noindent
The group $\Aut(G)$ also acts on $\mathcal{M}_3(G)$. The additional automorphisms in $\Aut(\mathcal{M}_3(G))$ are well-understood. 
They arise from swapping the two degree-$2$ endvertices of the attached triangles and from permuting
the triangles attached to a given vertex of graph~$G$. 

As mentioned above, Sabidussi showed for a range of properties, that they can be required 
of the graph that realises a given finite group. Theorem 3.7 in~\cite{Sabidussi1957} is more general than we need here;
in a version tailored for our purposes, it is:

\begin{theorem}[\cite{Sabidussi1957}]
\label{thm:sabidussi}
For every finite group $\mathfrak{G}$ of order $|\mathfrak{G}| > 1$ and $d \geq 3$ there exist infinitely many connected $d$-regular graphs $G$, such that $\Aut(G) \cong \mathfrak{G}$.
\end{theorem}

\noindent
The theorem of Sabidussi requires the group $\mathfrak{G}$ to be non-trivial. However, as Bollob{\'a}s has shown, a consequence of \cite[Theorem 6]{Bollobas1982} is that for $d \geq 3$ 
almost every $d$-regular graph is asymmetric.
Thus it is easy to incorporate the trivial case into Theorem~\ref{thm:sabidussi} and the requirement $|\mathfrak{G}| > 1$ could be omitted.
Theorem~\ref{thm:sabidussi} is the jumping-off point for our proofs.

\section{Proof of Theorem~\ref{thm:1}}

We are now ready to prove the main theorem. We will exploit a combination of the triangle-multiplier and coalescence constructions.

\begin{proof}[Proof of Theorem~\ref{thm:1}]
If $|\mathfrak{G}| > 1$, then by Theorem~\ref{thm:sabidussi}, there exists a $4$-regular graph $H$, such that $\Aut(H) \cong \mathfrak{G}$.
In the case $|\mathfrak{G}| = 1$, simply take $H$ to be the graph from Figure~3(a), i.e.\ an asymmetric 4-regular graph of the minimum order.
By Proposition~\ref{prop:multiplier3}, the graph $\mathcal{M}_3(H)$ is a nut graph such that $\Aut(H) \leq \Aut(\mathcal{M}_3(H))$.
By Proposition~\ref{prop:multiConstrSym}, $|\Aut(\mathcal{M}_3(H))| = 8^{|V(H)|} |\Aut(H)|$. 

Let us denote $\kappa = |V(H)|$ and $V(H) = \{h_1, h_2, \ldots, h_\kappa \}$. By definition, $H \subset \mathcal{M}_3(H)$. Let the
extra vertices be denoted $t_i^{(j, k)}$ for $1 \leq i \leq \kappa$ and $j, k \in \{1, 2\}$ such that the new
neighbours of $h_i$ are $t_i^{(1, 1)}, t_i^{(1, 2)}, t_i^{(2, 1)}$ and $t_i^{(2, 2)}$. Moreover, $t_i^{(j, 1)}$ and $t_i^{(j, 2)}$ are adjacent; see Figure~\ref{fig:hMulti}.
\begin{figure}[!htbp]
\centering
\begin{tikzpicture}
\tikzstyle{vertex}=[draw,circle,font=\scriptsize,minimum size=4pt,inner sep=1pt,fill=black]
\tikzstyle{bvertex}=[draw,circle,font=\scriptsize,minimum size=4pt,inner sep=1pt,color=blue,fill=blue!40!white]
\tikzstyle{edge}=[draw,thick]
\draw[thick, color=green!60!black, fill=green!20!white] (1,0)  ellipse (2.5 and 1.2);
\node[vertex,label=0:$h_1$] (h1) at (-1, 0) {};
\node[vertex,label=-30:$h_2$] (h2) at (0.8, 0.7) {};
\node at (3.4, 0.9) {$H$};
\node[bvertex,label=180:$t_1^{(1, 1)}$] (t1_11) at (-1.6, -1.0) {};
\node[bvertex,label=180:$t_1^{(1, 2)}$] (t1_12) at (-1.8, -0.3) {};
\node[bvertex,label=180:$t_1^{(2, 1)}$] (t1_21) at (-1.8, 0.3) {};
\node[bvertex,label=180:$t_1^{(2, 2)}$] (t1_22) at (-1.6, 1.0) {};
\path[edge,color=blue!70!white] (h1) -- (t1_11) --  (t1_12) -- (h1);
\path[edge,color=blue!70!white] (h1) -- (t1_21) --  (t1_22) -- (h1);
\node[bvertex,label=180:$t_2^{(1, 1)}$] (t2_11) at (0.0, 1.4) {};
\node[bvertex,label=90:$t_2^{(1, 2)}$] (t2_12) at (0.6, 1.6) {};
\path[edge,color=blue!70!white] (h2) -- (t2_11) --  (t2_12) -- (h2);
\node[bvertex,label=90:$t_2^{(2, 1)}$] (t2_21) at (1.5, 1.6) {};
\node[bvertex,label=0:$t_2^{(2, 2)}$] (t2_22) at (2.0, 1.4) {};
\path[edge,color=blue!70!white] (h2) -- (t2_21) --  (t2_22) -- (h2);
\node[draw=none,fill=none] (dotz) at (0.6, -0.5) {$\cdots$};
\path[edge] (h1) -- ($ (h1) + (100:0.4) $);
\path[edge] (h1) -- ($ (h1) + (70:0.4) $);
\path[edge] (h1) -- ($ (h1) + (-40:0.4) $);
\path[edge] (h1) -- ($ (h1) + (-80:0.4) $);
\path[edge] (h2) -- ($ (h2) + (-10:0.4) $);
\path[edge] (h2) -- ($ (h2) + (-80:0.4) $);
\path[edge] (h2) -- ($ (h2) + (-140:0.4) $);
\path[edge] (h2) -- ($ (h2) + (-170:0.4) $);
\end{tikzpicture}
\caption{The graph $\mathcal{M}_3(H)$.}
\label{fig:hMulti}
\end{figure}
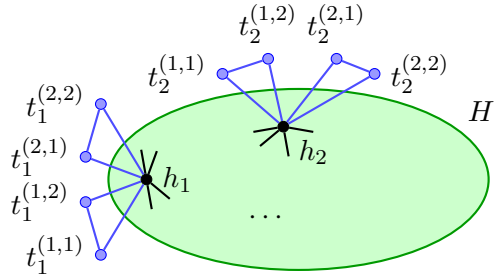

The automorphisms of $\mathcal{M}_3(H)$ are well-understood. Every $\alpha \in \Aut(H)$ is extended to an automorphism $\widehat{\alpha} \in \Aut(\mathcal{M}_3(H))$
by the following natural definition:
\begin{equation}
\widehat{\alpha}(v) = 
\begin{cases}
\alpha(v), & \text{ if } v \in V(H); \\
t_\ell^{(j, k)}, & \text{ if } v = t_i^{(j, k)} \text{ and } h_\ell = \alpha(h_i).
\end{cases}
\end{equation}
In addition to $\widehat{\alpha}$ for $\alpha \in \Aut(H)$, there are the following extra automorphisms in $ \Aut(\mathcal{M}_3(H))$:
\begin{align}
\beta_{i, j} & = ( t_i^{(j, 1)} \ t_i^{(j, 2)} ), \\
\gamma_i & = ( t_i^{(1, 1)} \ t_i^{(2, 1)} ) ( t_i^{(1, 2)} \ t_i^{(2, 2)} ),
\end{align}
for $i = 1, \ldots, \kappa$ and $j = 1, 2$.

We will remove the extra automorphisms by attaching `gadgets' to vertices $t_i^{(1, 1)}$ and $t_i^{(2, 1)}$ for $i = 1, \ldots, \kappa$. 
\begin{figure}[!htbp]
\centering
\begin{tikzpicture}
\tikzstyle{vertex}=[draw,circle,font=\scriptsize,minimum size=6pt,inner sep=1pt,fill=black]
\tikzstyle{bvertex}=[draw,circle,font=\scriptsize,minimum size=6pt,inner sep=1pt,color=blue,fill=blue!40!white]
\tikzstyle{edge}=[draw,thick]
\node[vertex,label=-90:$q_1$] (q1) at (-1, 0) {};
\node[vertex,label=-0:$q_2$] (q2) at (0.8, 0.7) {};
\node[bvertex] (v3) at (0.2, 1.4) {};
\node[bvertex] (v6) at (0.5, 2.0) {};
\node[bvertex] (v4) at (-1, 1.8) {};
\node[bvertex] (v1) at (-1.3, 1.1) {};
\node[bvertex] (v7) at (-0.7, 1.3) {};
\node[bvertex] (v2) at (-0.3, 0.7) {};
\path[edge] (v4) -- (v6) -- (q2) -- (q1) -- (v1) -- (v4);
\path[edge] (q1) -- (v7) -- (v6) -- (v3) -- (q2) -- (v7) -- (v2) -- (q1);
\end{tikzpicture}
\caption{The gadget graph $Q_0$. The root vertices used in the first and second attachment are labelled $q_1$ and $q_2$, respectively.}
\label{fig:gadget}
\end{figure}
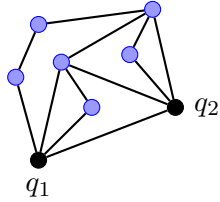
Consider the graph $Q_0$ in Figure~\ref{fig:gadget}. It is easy to verify that $Q_0$ is a nut graph of order~$8$ with $|\Aut(Q_0)| = 2$, and that vertices labelled $q_1$ and $q_2$
belong to different vertex orbits. Moreover, the respective stabilisers $\Aut(Q_0)_{q_1}$ and $\Aut(Q_0)_{q_2}$ are trivial. 

Let $G$ be the graph obtained from $\mathcal{M}_3(H)$ by a series of coalescence constructions. Start with $G_0 \coloneqq \mathcal{M}_3(H)$.
For $i=1, \ldots, \kappa$ define $G_i \coloneqq (G_{i - 1}, t_i^{(1, 1)}) \odot (Q_0, q_1)$. (The graph $G_i$ is obtained from $G_{i - 1}$ by adding a new copy of $Q_0$ to $G_{i - 1}$ and identifying 
$q_1$ with the vertex $t_i^{(1, 1)}$.) For $i=1, \ldots, \kappa$ define $G_{i + \kappa} \coloneqq (G_{i + \kappa - 1}, t_i^{(2, 1)}) \odot (Q_0, q_2)$. 
By Lemma~\ref{lem:coalescence}, $G_1, G_2, \ldots, G_{2\kappa}$ are all nut graphs. Let $G \coloneqq G_{2\kappa}$.

Next, observe that automorphisms $\widehat{\alpha}$ can be extended naturally from $\mathcal{M}_3(H)$ to $G$.
However, all automorphisms $\beta_{i, j}$ have been removed, since vertices $t_i^{(j, 1)}$ now carry gadgets, while vertices $t_i^{(j, 2)}$ do not (they are still of degree $2$).
Similarly, all automorphisms $\gamma_i$ have been removed, since the gadget attached to  $t_i^{(1, 1)}$ does not map to the gadget attached to $t_i^{(1, 2)}$, as vertices $q_1$
and $q_2$ are in different vertex orbits of $Q_0$. Moreover, no new automorphisms have been introduced, as vertices $q_1, q_2 \in V(Q_0)$ have trivial stabilisers.
Therefore, $\Aut(G) \cong \Aut(H) \cong \mathfrak{G}$. 

We provided one nut graph $G$ which realises the group $\mathfrak{G}$.  To obtain an infinite family, we can subdivide each edge from $\{ h_i t_i^{(1, 2)} \mid i = 1, \ldots, \kappa \}$
with $4\sigma$ vertices for any choice of $\sigma \geq 0$, i.e.\ we use the subdivision construction on these edges. 
\end{proof}

Note that there are many `degrees of freedom' in the proof of Theorem~\ref{thm:1}. 
In our construction, we could have taken $H$ to be \emph{any} $4$-regular graph that realises the given group $\mathfrak{G}$. 
In case $|\mathfrak{G}| > 1$, Theorem~\ref{thm:sabidussi} already provides infinitely many starting graphs $H$ (which in turn produce infinitely many non-isomorphic nut graphs $G$).
If  $|\mathfrak{G}| = 1$, by~\cite{Bollobas1982}, there are also infinitely many startings graphs $H$.
At the coalescence stage, we could have picked different vertices as $q_1$ and $q_2$ in $Q_0$ (so long as they are in different vertex orbits). We could also have choosen a different gadget graph for $Q_0$, or taken two different gadget graphs. We could have decorated both triangles with the same gadget and taken $q_1 = q_2$; that choice would have removed only elements $\beta_{i, j}$; to further remove elements $\gamma_i$, we could have used the subdivision construction on edges $h_i t_i^{(1, 2)}$.

The multiplier-coalescence construction is prodigal in terms of the number of vertices of the nut graphs obtained.
The order of graph $G$ provided by the proof of Theorem~\ref{thm:1} is $19|V(H)|$, where $|V(H)|$ is the order of the graph $H$.
For a given group $\mathfrak{G}$, $|\mathfrak{G}| > 3$, with $\nu$ generators, the smallest $4$-regular graph of the family constructed by Sabidussi in \cite{Sabidussi1957} is of order $4(\nu + 2)|\mathfrak{G}|$. Therefore, the order of the smallest graph obtained from Sabidussi's starting graph is $76 (\nu + 2)| \mathfrak{G}|$. 

\begin{figure}[!htb]
\centering
\subcaptionbox{\label{subfig:3a}}
{ \begin{tikzpicture}[scale=1.0]
\tikzstyle{vertex}=[draw,circle,font=\scriptsize,minimum size=6pt,inner sep=1pt,fill=blue!40!white]
\tikzstyle{edge}=[draw,thick]
\node[vertex] (v1) at (0, -0.2) {};
\node[vertex] (v2) at (0, 1.2) {};
\node[vertex] (v3) at (0.5, 1.7) {};
\node[vertex] (v4) at (1.7, 1.7) {};
\node[vertex] (v5) at (2.2, 1.2) {};
\node[vertex] (v6) at (2.2, -0.2) {};
\node[vertex] (v7) at (1.7, -0.9) {};
\node[vertex] (v8) at (0.5, -0.9) {};
\node[vertex] (v9) at (1.7, 0.7) {};
\node[vertex] (v10) at (0.5, 0.7) {};
\path[edge] (v1) -- (v2) -- (v3) -- (v4) -- (v5) -- (v6) -- (v7) -- (v8) -- (v1);
\path[edge] (v2) -- (v10) -- (v9) -- (v5);
\path[edge] (v1) -- (v10) -- (v8) -- (v4);
\path[edge] (v1) -- (v6) -- (v9);
\path[edge] (v7) -- (v3) -- (v9);
\path[edge] (v2) -- (v4);
\path[edge] (v5) -- (v7);
\end{tikzpicture} }
\qquad
\subcaptionbox{\label{subfig:3b}}
{ \begin{tikzpicture}[scale=1.0]
\tikzstyle{vertex}=[draw,circle,font=\scriptsize,minimum size=6pt,inner sep=1pt,fill=blue!40!white]
\tikzstyle{edge}=[draw,thick]
\node[vertex] (v1) at (-1.1, 0) {};
\node[vertex] (u1) at (1.1, 0) {};
\node[vertex] (v2) at (-0.3, -0.7) {};
\node[vertex] (u2) at (0.3, -0.7) {};
\node[vertex] (v3) at (-0.5, -1.7) {};
\node[vertex] (u3) at (0.5, -1.7) {};
\node[vertex] (v4) at (-1.5, -2.4) {};
\node[vertex] (u4) at (1.5, -2.4) {};
\node[vertex] (w) at (0, -3.0) {};
\path[edge] (u1) -- (u4) -- (w) -- (v4) -- (v1);
\path[edge] (u1) -- (u3) -- (w) -- (v3) -- (v1);
\path[edge] (u1) -- (u2) -- (v3) -- (v4);
\path[edge] (v1) -- (v2) -- (u3) -- (u4);
\path[edge] (v1) -- (u1);
\path[edge] (v4) -- (v2) -- (u2) -- (u4);
\end{tikzpicture} }
\qquad
\subcaptionbox{\label{subfig:3c}}
{ \begin{tikzpicture}[scale=1.0]
\tikzstyle{vertex}=[draw,circle,font=\scriptsize,minimum size=6pt,inner sep=1pt,fill=blue!40!white]
\tikzstyle{edge}=[draw,thick]
\node[vertex] (v1) at (70:0.9) {};
\node[vertex] (v2) at (75+5:1.8) {};
\node[vertex] (v3) at (110:0.8) {};
\node[vertex] (v4) at (105+5:1.8) {};
\node[vertex] (u1) at ({70+120}:0.9) {};
\node[vertex] (u2) at ({75+120+5}:1.8) {};
\node[vertex] (u3) at ({110+120}:0.8) {};
\node[vertex] (u4) at ({105+120+5}:1.8) {};
\node[vertex] (w1) at ({70-120}:0.9) {};
\node[vertex] (w2) at ({75-120+5}:1.8) {};
\node[vertex] (w3) at ({110-120}:0.8) {};
\node[vertex] (w4) at ({105-120+5}:1.8) {};
\node[vertex] (a) at (0.15, 0.15) {};
\node[vertex] (b) at (-0.15, -0.15) {};
\path[edge] (b) -- (a);
\path[edge] (v1) -- (a);
\path[edge] (u1) -- (a);
\path[edge] (w1) -- (a);
\path[edge] (v3) -- (b);
\path[edge] (u3) -- (b);
\path[edge] (w3) -- (b);
\path[edge] (u4) -- (u3) -- (u1) -- (u2) -- (u4);
\path[edge] (v4) -- (v3) -- (v1) -- (v2) -- (v4);
\path[edge] (w4) -- (w3) -- (w1) -- (w2) -- (w4);
\path[edge] (v2) -- (w2) -- (u2) -- (v2);
\path[edge] (v4) -- (w3);
\path[edge] (u4) -- (v3);
\path[edge] (w4) -- (u3);
\path[edge] (v4) -- (u1);
\path[edge] (u4) -- (w1);
\path[edge] (w4) -- (v1);
\end{tikzpicture} }
\caption{Graphs that realise the minimum order among $4$-regular graphs with automorphism groups $\mathbb{Z}_1, \mathbb{Z}_2$ and 
$\mathbb{Z}_3$, respectively. These graphs are not uniquely determined; they are selected from sets of $4$, $3$ and $8$ candidates, respectively.}
\label{ref:smallestFourReg}
\end{figure}
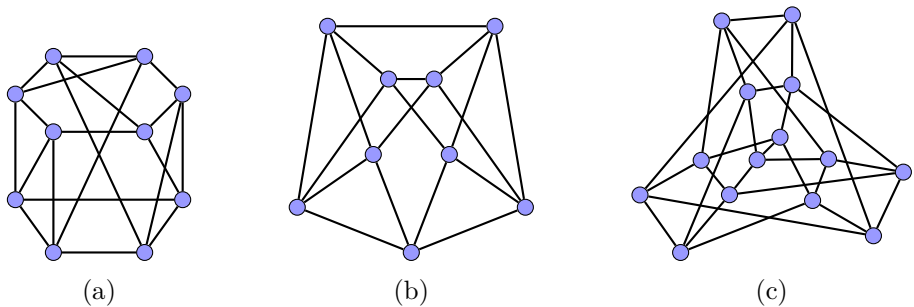

Typically, much smaller examples can exist. Instead of the graph provided by the construction in the proof by Sabidussi, we could take the starting graph $H$ to be a minimal $4$-regular graph
that realises~$\mathfrak{G}$.
For groups $\mathbb{Z}_1, \mathbb{Z}_2$ and 
$\mathbb{Z}_3$, minimal graphs $H$ are shown in Figure~\ref{ref:smallestFourReg}. These have orders $10, 9$ and $14$, respectively. Application of the multiplier-coalescence
construction gives rise to
nut graphs of respective orders $190, 171$ and $266$.

\section{Proof of Theorem~\ref{thm:2}}

\begin{proof}[Proof of Theorem~\ref{thm:2}]
The proof proceeds as for Theorem~\ref{thm:1} to the point where gadgets are attached to $\mathcal{M}_3(H)$.

First, we prove the case $d = 8$.
Consider the graphs $P_1, P_2$ and $P_3$ in Figure~\ref{fig:protoGadgets}. They are non-isomorphic graphs; each of them contains six degree-$3$ vertices and six degree-$4$ vertices.
The gadget $Q_i$, $1 \leq i \leq 3$, is obtained from $P_i$ by adding a new vertex $w_i$ to its complement $\overline{P_i}$ and joining $w_i$ to all degree-$7$ vertices of~$\overline{P_i}$.
Observe that $Q_1, Q_2$ and $Q_3$ are non-isomorphic graphs of order $13$. All vertices of $Q_i$ are of degree $8$, except for $w_i$ which is of degree $6$. It is easy to verify
that $Q_1, Q_2$ and $Q_3$ are nut graphs and that their automorphism groups are trivial.
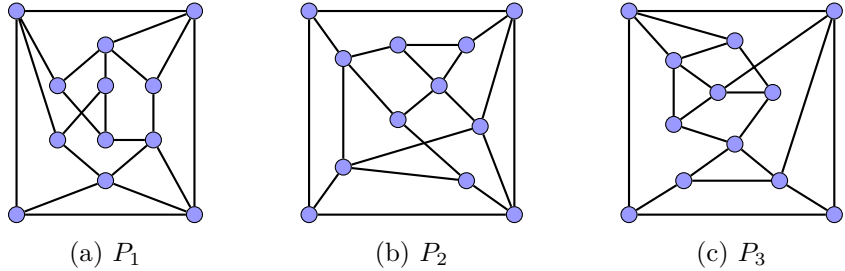
\begin{figure}[!htb]
\centering
\subcaptionbox{\label{subfig:4a}$P_1$}
{ \begin{tikzpicture}[scale=0.9]
\tikzstyle{vertex}=[draw,circle,font=\scriptsize,minimum size=6pt,inner sep=1pt,fill=blue!40!white]
\tikzstyle{edge}=[draw,thick]
\node[vertex] (v5) at (-0.3, 0) {};
\node[vertex] (v6) at (2.3, 0) {};
\node[vertex] (v3) at (-0.3, -3) {};
\node[vertex] (v1) at (2.3, -3) {};
\node[vertex] (v2) at (1, -2.5) {};
\node[vertex] (v7) at (1, -0.5) {};
\node[vertex] (v11) at (1.7, -1.1) {};
\node[vertex] (v9) at (1, -1.1) {};
\node[vertex] (v8) at (0.3, -1.1) {};
\node[vertex] (v0) at (1.7, -1.9) {};
\node[vertex] (v10) at (1, -1.9) {};
\node[vertex] (v4) at (0.3, -1.9) {};
\path[edge] (v5) -- (v3) -- (v1) -- (v6) -- (v5);
\path[edge] (v3) -- (v2) -- (v1) -- (v0) -- (v10) -- (v8) -- (v7) -- (v6) -- (v11) -- (v7) -- (v9) -- (v10);
\path[edge] (v11) -- (v0) -- (v2) -- (v4) -- (v9);
\path[edge] (v4) -- (v5) -- (v8);
\end{tikzpicture} }
\qquad
\subcaptionbox{\label{subfig:4b}$P_2$}
{ \begin{tikzpicture}[scale=0.9]
\tikzstyle{vertex}=[draw,circle,font=\scriptsize,minimum size=6pt,inner sep=1pt,fill=blue!40!white]
\tikzstyle{edge}=[draw,thick]
\node[vertex] (v3) at (0, 0) {};
\node[vertex] (v2) at (3, 0) {};
\node[vertex] (v4) at (0, -3) {};
\node[vertex] (v1) at (3, -3) {};
\node[vertex] (v6) at (0.5, -0.7) {};
\node[vertex] (v10) at (2.3, -0.5) {};
\node[vertex] (v5) at (0.5, -2.3) {};
\node[vertex] (v7) at (2.3, -2.5) {};
\node[vertex] (v8) at (1.3, -1.6) {};
\node[vertex] (v11) at (1.3, -0.5) {};
\node[vertex] (v0) at (2.5, -1.7) {};
\node[vertex] (v9) at (1.9, -1.1) {};
\path[edge] (v3) -- (v4) -- (v1) -- (v2) -- (v3);
\path[edge] (v3) -- (v6) -- (v5) -- (v4);
\path[edge] (v1) -- (v7) -- (v5) -- (v0) -- (v1);
\path[edge] (v7) -- (v8) -- (v9) -- (v0) -- (v2) -- (v10) -- (v11) -- (v9) -- (v10);
\path[edge] (v11) -- (v6) -- (v8);
\end{tikzpicture} }
\qquad
\subcaptionbox{\label{subfig:4c}$P_3$}
{ \begin{tikzpicture}[scale=0.9]
\tikzstyle{vertex}=[draw,circle,font=\scriptsize,minimum size=6pt,inner sep=1pt,fill=blue!40!white]
\tikzstyle{edge}=[draw,thick]
\node[vertex] (v5) at (0, 0) {};
\node[vertex] (v0) at (3, 0) {};
\node[vertex] (v4) at (0, -3) {};
\node[vertex] (v1) at (3, -3) {};
\node[vertex] (v3) at (0.8, -2.5) {};
\node[vertex] (v2) at (2.2, -2.5) {};
\node[vertex] (v7) at (1.3, -1.2) {};
\node[vertex] (v10) at ($ (v7) + (0:0.8) $) {};
\node[vertex] (v8) at ($ (v7) + (72:0.8) $) {};
\node[vertex] (v6) at ($ (v7) + ({72*2}:0.8) $) {};
\node[vertex] (v11) at ($ (v7) + ({72*3}:0.8) $) {};
\node[vertex] (v9) at ($ (v7) + ({72*4}:0.8) $) {};
\path[edge] (v5) -- (v0) -- (v1) -- (v4) -- (v5);
\path[edge] (v4) -- (v3) -- (v2) -- (v0) -- (v7) -- (v6) -- (v8) -- (v10) -- (v9) -- (v11) -- (v6);
\path[edge] (v8) -- (v5) -- (v6);
\path[edge] (v10) -- (v7) -- (v11);
\path[edge] (v3) -- (v9) -- (v2) -- (v1);
\end{tikzpicture} }
\caption{The proto-gadget graphs for the proof of Theorem~\ref{thm:2} in the case $d = 8$.}
\label{fig:protoGadgets}
\end{figure}

As in the proof of Theorem~\ref{thm:1}, we obtain $G$ by a series of coalescence constructions. Start with $G_0 \coloneqq \mathcal{M}_3(H)$.
For $i=1, \ldots, \kappa$ define $G_i \coloneqq (G_{i - 1}, t_i^{(1, 1)}) \odot (Q_1, w_1)$. For $i=1, \ldots, \kappa$ define $G_{i + \kappa} \coloneqq (G_{i + \kappa - 1}, t_i^{(2, 1)}) \odot (Q_1, w_1)$.
For $i=1, \ldots, \kappa$ define $G_{i + 2\kappa} \coloneqq (G_{i + 2\kappa - 1}, t_i^{(1, 2)}) \odot (Q_2, w_2)$.
For $i=1, \ldots, \kappa$ define $G_{i + 3\kappa} \coloneqq (G_{i + 3\kappa - 1}, t_i^{(2, 2)}) \odot (Q_3, w_3)$.
In other words, gadgets $Q_1, Q_2$ and $Q_3$ are attached to degree-$2$ vertices of the triangles as indicated schematically in Figure~\ref{fig:schematicTriangles}(a).
By Lemma~\ref{lem:coalescence}, $G_1, G_2, \ldots, G_{4\kappa}$ are all nut graphs. Let $G \coloneqq G_{4\kappa}$.

By similar reasoning to that used in the proof of Theorem~\ref{thm:1}, we can see that automorphisms $\widehat{\alpha}$ can be extended naturally from $\mathcal{M}_3(H)$ to $G$.
Moreover, the gadgets $Q_1, Q_2$ and $Q_3$ were attached in a manner such that automorphisms $\beta_{i, j}$ and $\gamma_i$ were removed.
Further, attachment has introduced no new automorphisms, as these gadgets  all have trivial symmetry.
Hence, $\Aut(G) \cong \Aut(H) \cong \mathfrak{G}$. Finally, observe that all vertices of $G$ are of degree $8$.
This proves the case $d = 8$.
For higher values of $d$ the proof is similar, but 
the search for the requisite number of proto-gadgets becomes rapidly more tedious.

To prove the result for a given $d$, we start with a $(d/2)$-regular graph $H$ that realises the group $\mathfrak{G}$.
If $|\mathfrak{G}| > 1$, Theorem~\ref{thm:sabidussi} provides us with infinitely many such graphs $H$.
If $|\mathfrak{G}| = 1$, by \cite{Bollobas1982}, there are also infinitely many such graphs $H$.
By Proposition~\ref{prop:multiplier3}, $\mathcal{M}_3(H)$ is a nut graph. In this graph, there are $d / 4$
triangles attached at every vertex of $H$. To remove the unwanted symmetries, every triangle is
decorated by a different pair of gadgets. (See Figure~\ref{fig:schematicTriangles}.) With $s$ gadgets, we can form $\binom{s}{2}$ different pairs. 
We choose the smallest $s$ such that $\binom{s}{2} \geq d/4$.
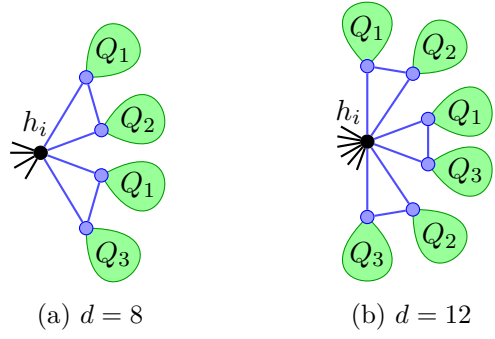
\begin{figure}[!htbp]
\centering
\subcaptionbox{\label{subfig:5a}$d = 8$}
{ \begin{tikzpicture}
\tikzstyle{baloon}=[draw,color=green!60!black,fill=green!40!white]
\tikzstyle{vertex}=[draw,circle,font=\scriptsize,minimum size=5pt,inner sep=1pt,fill=black]
\tikzstyle{bvertex}=[draw,circle,font=\scriptsize,minimum size=5pt,inner sep=1pt,color=blue,fill=blue!40!white]
\tikzstyle{edge}=[draw,thick]
\node[vertex,label={[xshift=-2pt]90:$h_i$}] (h1) at (1, 0) {};
 \clip (0.6,-1.75) rectangle (2.7,1.75);
\coordinate (c1_11) at (1.6, -1.0);
\coordinate (c1_12) at (1.8, -0.3);
\coordinate (c1_21) at (1.8, 0.3);
\coordinate (c1_22) at (1.6, 1.0);
\draw[baloon] (c1_11) .. controls ($ (c1_11) + (0:1.6) $) and ($ (c1_11) + (-90:1.6) $) .. (c1_11);
\draw[baloon] (c1_12) .. controls ($ (c1_12) + (30:1.6) $) and ($ (c1_12) + (-60:1.6) $) .. (c1_12);
\draw[baloon] (c1_21) .. controls ($ (c1_21) + (60:1.6) $) and ($ (c1_21) + (-30:1.6) $) .. (c1_21);
\draw[baloon] (c1_22) .. controls ($ (c1_22) + (0:1.6) $) and ($ (c1_22) + (90:1.6) $) .. (c1_22);
\node[bvertex] (t1_11) at (c1_11) {};
\node[bvertex] (t1_12) at (c1_12) {};
\node[bvertex] (t1_21) at (c1_21) {};
\node[bvertex] (t1_22) at (c1_22) {};
\path[edge,color=blue!70!white] (h1) -- (t1_11) --  (t1_12) -- (h1);
\path[edge,color=blue!70!white] (h1) -- (t1_21) --  (t1_22) -- (h1);
\path[edge] (h1) -- ($ (h1) + (160:0.4) $);
\path[edge] (h1) -- ($ (h1) + (180:0.4) $);
\path[edge] (h1) -- ($ (h1) + (210:0.4) $);
\path[edge] (h1) -- ($ (h1) + (240:0.4) $);
\node[fill=none,draw=none] at ($ (c1_11) + (-45:0.5) $) {$Q_3$};
\node[fill=none,draw=none] at ($ (c1_12) + (-15:0.5) $) {$Q_1$};
\node[fill=none,draw=none] at ($ (c1_21) + (15:0.5) $) {$Q_2$};
\node[fill=none,draw=none] at ($ (c1_22) + (45:0.5) $) {$Q_1$};
\end{tikzpicture} }
\qquad\qquad
\subcaptionbox{\label{subfig:5b}$d = 12$}
{ \begin{tikzpicture}
\tikzstyle{baloon}=[draw,color=green!60!black,fill=green!40!white]
\tikzstyle{vertex}=[draw,circle,font=\scriptsize,minimum size=5pt,inner sep=1pt,fill=black]
\tikzstyle{bvertex}=[draw,circle,font=\scriptsize,minimum size=5pt,inner sep=1pt,color=blue,fill=blue!40!white]
\tikzstyle{edge}=[draw,thick]
\node[vertex,label={[xshift=-7pt]90:$h_i$}] (h1) at (1, 0) {};
\clip (0.5,-1.9) rectangle (2.7,1.9);
\coordinate (c1_11) at (1.6, -0.9);
\coordinate (c1_12) at (1.0, -1.0);
\coordinate (c1_21) at (1.8, -0.3);
\coordinate (c1_22) at (1.8, 0.3);
\coordinate (c1_31) at (1.6, 0.9);
\coordinate (c1_32) at (1.0, 1.0);
\draw[baloon] (c1_31) .. controls ($ (c1_31) + (90:1.6) $) and ($ (c1_31) + (0:1.6) $) .. (c1_31);
\node[fill=none,draw=none] at ($ (c1_31) + (45:0.5) $) {$Q_2$};
\draw[baloon] (c1_32) .. controls ($ (c1_32) + (135:1.6) $) and ($ (c1_32) + (45:1.6) $) .. (c1_32);
\node[fill=none,draw=none] at ($ (c1_32) + (90:0.5) $) {$Q_1$};
\draw[baloon] (c1_11) .. controls ($ (c1_11) + (0:1.6) $) and ($ (c1_11) + (-90:1.6) $) .. (c1_11);
\node[fill=none,draw=none] at ($ (c1_11) + (-45:0.5) $) {$Q_2$};
\draw[baloon] (c1_12) .. controls ($ (c1_12) + (-135:1.6) $) and ($ (c1_12) + (-45:1.6) $) .. (c1_12);
\node[fill=none,draw=none] at ($ (c1_12) + (-90:0.5) $) {$Q_3$};
\draw[baloon] (c1_21) .. controls ($ (c1_21) + (35:1.6) $) and ($ (c1_21) + (-55:1.6) $) .. (c1_21);
\node[fill=none,draw=none] at ($ (c1_21) + (-10:0.5) $) {$Q_3$};
\draw[baloon] (c1_22) .. controls ($ (c1_22) + (55:1.6) $) and ($ (c1_22) + (-35:1.6) $) .. (c1_22);
\node[fill=none,draw=none] at ($ (c1_22) + (10:0.5) $) {$Q_1$};
\node[bvertex] (t1_11) at (c1_11) {};
\node[bvertex] (t1_12) at (c1_12) {};
\node[bvertex] (t1_21) at (c1_21) {};
\node[bvertex] (t1_22) at (c1_22) {};
\node[bvertex] (t1_31) at (c1_31) {};
\node[bvertex] (t1_32) at (c1_32) {};
\path[edge,color=blue!70!white] (h1) -- (t1_11) --  (t1_12) -- (h1);
\path[edge,color=blue!70!white] (h1) -- (t1_21) --  (t1_22) -- (h1);
\path[edge,color=blue!70!white] (h1) -- (t1_31) --  (t1_32) -- (h1);
\path[edge] (h1) -- ($ (h1) + (150:0.4) $);
\path[edge] (h1) -- ($ (h1) + (170:0.4) $);
\path[edge] (h1) -- ($ (h1) + (190:0.4) $);
\path[edge] (h1) -- ($ (h1) + (210:0.4) $);
\path[edge] (h1) -- ($ (h1) + (230:0.4) $);
\path[edge] (h1) -- ($ (h1) + (250:0.4) $);
\end{tikzpicture} }
\caption{Arrangements of gadgets $Q_i$ on degree-$2$ vertices of a bouquet of triangles that remove unwanted automorphisms (for the cases $d = 8$ and $d = 12$).}
\label{fig:schematicTriangles}
\end{figure}
Figure~\ref{fig:protoGadgetsExtra} tabulates a sufficient set of proto-gadgets for degrees $12, 16, 20$ and $24$. The complement of $P_i^{(d)}$ contains $d - 2$ vertices
of degree $d - 1$, and the remaining vertices are of degree $d$. To obtain $Q_i^{(d)}$, add a new vertex to the complement of $P_i^{(d)}$ and connect it to all vertices of degree $d - 1$.
Graph $Q_i^{(d)}$ has exactly one vertex of degree $d - 1$, while the rest are of degree $d$. It is easy to verify that graphs $Q_1^{(d)}, Q_2^{(d)}, \ldots$ are non-isomorphic
and that they all have trivial symmetry.
\begin{figure}[!p]
\centering
\subcaptionbox{\label{subfig:6a}$P_1^{(12)}$}
{ \begin{tikzpicture}[scale=0.8,rotate=90]
\tikzstyle{vertex}=[draw,circle,font=\scriptsize,minimum size=6pt,inner sep=1pt,fill=blue!40!white]
\tikzstyle{edge}=[draw,thick]
\node[vertex] (v1) at (0, 0) {};
\node[vertex] (v3) at (3, 0) {};
\node[vertex] (v7) at (0, -4) {};
\node[vertex] (v6) at (3, -4) {};
\node[vertex] (v2) at (1.5, 0) {};
\node[vertex] (v9) at (0, -0.8) {};
\node[vertex] (v8) at (0, -1.6) {};
\node[vertex] (v4) at (3, -0.8) {};
\node[vertex] (v5) at (3, -1.6) {};
\node[vertex] (v10) at (1, -1.7) {};
\node[vertex] (v11) at (2, -1.7) {};
\node[vertex] (v12) at (1.5, -2.5) {};
\node[vertex] (v15) at (1.5, -3.3) {};
\node[vertex] (v13) at (0.8, -2.9) {};
\node[vertex] (v14) at (2.2, -2.9) {};
\path[edge] (v1) -- (v2) -- (v3) -- (v4) -- (v5) -- (v6) -- (v7) -- (v8) -- (v9) -- (v1);
\path[edge] (v9) -- (v4);
\path[edge] (v8) -- (v10) -- (v11) -- (v5);
\path[edge] (v11) -- (v12) -- (v13) -- (v15) -- (v14) -- (v12);
\path[edge] (v7) -- (v13);
\path[edge] (v6) -- (v14);
\end{tikzpicture} }
\subcaptionbox{\label{subfig:6b}$P_2^{(12)}$}
{ \begin{tikzpicture}[scale=0.8,rotate=90]
\tikzstyle{vertex}=[draw,circle,font=\scriptsize,minimum size=6pt,inner sep=1pt,fill=blue!40!white]
\tikzstyle{edge}=[draw,thick]
\node[vertex] (v1) at (0, 0) {};
\node[vertex] (v2) at (3, 0) {};
\node[vertex] (v6) at (0, -4) {};
\node[vertex] (v5) at (3, -4) {};
\node[vertex] (v3) at (3, -1.7) {};
\node[vertex] (v4) at (3, -3.2) {};
\node[vertex] (v7) at (0, -2) {};
\node[vertex] (v8) at (0.8, -0.8) {};
\node[vertex] (v9) at (1.5, -0.6) {};
\node[vertex] (v10) at (2.3, -0.8) {};
\node[vertex] (v14) at (1, -3) {};
\node[vertex] (v15) at (2, -3) {};
\node[vertex] (v12) at (1, -2.2) {};
\node[vertex] (v13) at (2, -2.2) {};
\node[vertex] (v11) at (1.3, -1.5) {};
\path[edge] (v1) -- (v2) -- (v3) -- (v4) -- (v5) -- (v6) -- (v7) -- (v1);
\path[edge] (v6) -- (v14) -- (v15) -- (v4);
\path[edge] (v14) -- (v12) -- (v11) -- (v8) -- (v1);
\path[edge] (v8) -- (v9) -- (v10) -- (v2);
\path[edge] (v10) -- (v13) -- (v15);
\path[edge] (v11) -- (v3);
\end{tikzpicture} }
\subcaptionbox{\label{subfig:6c}$P_3^{(12)}$}
{ \begin{tikzpicture}[scale=0.8,rotate=90]
\tikzstyle{vertex}=[draw,circle,font=\scriptsize,minimum size=6pt,inner sep=1pt,fill=blue!40!white]
\tikzstyle{edge}=[draw,thick]
\node[vertex] (v1) at (0, 0) {};
\node[vertex] (v3) at (3, 0) {};
\node[vertex] (v7) at (0, -4) {};
\node[vertex] (v5) at (3, -4) {};
\node[vertex] (v2) at (1.5, 0) {};
\node[vertex] (v4) at (3, -2) {};
\node[vertex] (v6) at (1.5, -4) {};
\node[vertex] (v8) at (0, -1.2) {};
\node[vertex] (v9) at (0.5, -0.7) {};
\node[vertex] (v10) at (1, -0.4) {};
\node[vertex] (v15) at (1.1, -3.0) {};
\node[vertex] (v14) at (0.4, -2.0) {};
\node[vertex] (v16) at (2.5, -2.7) {};
\node[vertex] (v13) at (2.0, -2.0) {};
\node[vertex] (v12) at (2.2, -1.2) {};
\node[vertex] (v11) at (2.6, -0.7) {};
\path[edge] (v1) -- (v2) -- (v3) -- (v4) -- (v5) -- (v6) -- (v7) -- (v8) -- (v1);
\path[edge] (v7) -- (v14) -- (v15) -- (v5);
\path[edge] (v8) -- (v9) -- (v10) -- (v3);
\path[edge] (v8) -- (v14) -- (v10) -- (v12) -- (v11) -- (v2) -- (v13) -- (v12) -- (v4) -- (v16);
\path[edge] (v9) -- (v13) -- (v6) -- (v16);
\path[edge] (v1) -- (v15) -- (v16) -- (v11);
\end{tikzpicture} }

\vspace{\baselineskip}
\subcaptionbox{\label{subfig:6d}$P_1^{(16)}$}
{ \begin{tikzpicture}[scale=0.8,rotate=90]
\tikzstyle{vertex}=[draw,circle,font=\scriptsize,minimum size=6pt,inner sep=1pt,fill=blue!40!white]
\tikzstyle{edge}=[draw,thick]
\node[vertex] (v1) at (0, 0) {};
\node[vertex] (v3) at (3, 0) {};
\node[vertex] (v10) at (0, -4) {};
\node[vertex] (v7) at (3, -4) {};
\node[vertex] (v2) at (1.0, 0) {};
\node[vertex] (v4) at (3, -1.4) {};
\node[vertex] (v5) at (3, -2) {};
\node[vertex] (v6) at (3, -2.6) {};
\node[vertex] (v8) at (2, -4) {};
\node[vertex] (v9) at (1, -4) {};
\node[vertex] (v11) at (1.2, -3.2) {};
\node[vertex] (v12) at (2.2, -3.2) {};
\node[vertex] (v13) at (1.7, -3.6) {};
\node[vertex] (v19) at (1.0, -2.4) {}; 
\node[vertex] (v14) at (0.8, -0.8) {}; 
\node[vertex] (v15) at (1.6, -1.8) {}; 
\node[vertex] (v16) at (1.8, -0.8) {}; 
\node[vertex] (v18) at (2.4, -0.3) {}; 
\node[vertex] (v17) at (2.4, -1.2) {}; 
\path[edge] (v1) -- (v2) -- (v3) -- (v4) -- (v5) -- (v6) -- (v7) -- (v8) -- (v9) -- (v10) -- (v1);
\path[edge] (v9) -- (v11) -- (v12) -- (v13) -- (v8);
\path[edge] (v11) -- (v13);
\path[edge] (v12) -- (v7);
\path[edge] (v2) -- (v14) -- (v19) -- (v6);
\path[edge] (v14) -- (v15) -- (v5); 
\path[edge] (v3) -- (v18) -- (v16) -- (v17) -- (v4);
\path[edge] (v15) -- (v16);
\end{tikzpicture} }
\subcaptionbox{\label{subfig:6e}$P_2^{(16)}$}
{ \begin{tikzpicture}[scale=0.8,rotate=90]
\tikzstyle{vertex}=[draw,circle,font=\scriptsize,minimum size=6pt,inner sep=1pt,fill=blue!40!white]
\tikzstyle{edge}=[draw,thick]
\node[vertex] (v1) at (0, 0) {};
\node[vertex] (v4) at (3, 0) {};
\node[vertex] (v10) at (0, -4) {};
\node[vertex] (v9) at (3, -4) {};
\node[vertex] (v2) at (1, 0) {};
\node[vertex] (v3) at (2, 0) {};
\node[vertex] (v5) at (3, -0.8) {};
\node[vertex] (v6) at (3, -1.6) {};
\node[vertex] (v7) at (3, -2.4) {};
\node[vertex] (v8) at (3, -3.2) {};
\node[vertex] (v11) at (0, -2.4) {};
\node[vertex] (v12) at (1, -3.4) {};
\node[vertex] (v13) at (2, -3.4) {};
\node[vertex] (v14) at (1.5, -2.7) {};
\node[vertex] (v15) at (1.3, -1.4) {};
\node[vertex] (v16) at (2.0, -2.1) {};
\node[vertex] (v17) at (2.0, -1.2) {};
\node[vertex] (v19) at (2.4, -0.5) {};
\node[vertex] (v18) at (1.7, -0.5) {};
\path[edge] (v1) -- (v2) -- (v3) -- (v4) -- (v5) -- (v6) -- (v7) -- (v8) -- (v9) -- (v10) -- (v11) -- (v1);
\path[edge] (v12) -- (v13) -- (v14) -- (v12) -- (v10);
\path[edge] (v13) -- (v9);
\path[edge] (v14) -- (v8);
\path[edge] (v7) -- (v16) -- (v15) -- (v2);
\path[edge] (v15) -- (v17) -- (v6); 
\path[edge] (v17) -- (v18) -- (v3); 
\path[edge] (v18) -- (v19) -- (v5); 
\end{tikzpicture} }
\subcaptionbox{\label{subfig:6f}$P_3^{(16)}$}
{ \begin{tikzpicture}[scale=0.8,rotate=90]
\tikzstyle{vertex}=[draw,circle,font=\scriptsize,minimum size=6pt,inner sep=1pt,fill=blue!40!white]
\tikzstyle{edge}=[draw,thick]
\node[vertex] (v1) at (0, 0) {};
\node[vertex] (v4) at (3, 0) {};
\node[vertex] (v7) at (0, -4) {};
\node[vertex] (v6) at (3, -4) {};
\node[vertex] (v2) at (1, 0) {};
\node[vertex] (v3) at (2, 0) {};
\node[vertex] (v5) at (3, -1.5) {};
\node[vertex] (v8) at (0, -3.0) {};
\node[vertex] (v9) at (0, -1.3) {};
\node[vertex] (v11) at (2.5, -0.8) {};
\node[vertex] (v10) at (2.5, -3.2) {};
\node[vertex] (v12) at (1.1, -0.6) {};
\node[vertex] (v13) at (1.9, -0.6) {};
\node[vertex] (v14) at (1.5, -1.0) {};
\node[vertex] (v15) at (1.8, -2.2) {};
\node[vertex] (v16) at (0.8, -3.2) {};
\node[vertex] (v17) at (1.5, -2.8) {};
\node[vertex] (v18) at (0.5, -2.3) {};
\node[vertex] (v19) at (1.0, -1.5) {};
\path[edge] (v7) -- (v16) -- (v17) -- (v15); 
\path[edge] (v2) -- (v12) -- (v13) -- (v3); 
\path[edge] (v1) -- (v2) -- (v3) -- (v4) -- (v5) -- (v6) -- (v7) -- (v8) -- (v9) -- (v1); 
\path[edge] (v4) -- (v11) -- (v10) -- (v6); 
\path[edge] (v12) -- (v14) -- (v13); 
\path[edge] (v14) -- (v15) -- (v10); 
\path[edge] (v8) -- (v18) -- (v19) -- (v16); 
\path[edge] (v1) -- (v19);
\end{tikzpicture} }
\subcaptionbox{\label{subfig:6g}$P_4^{(16)}$}
{ \begin{tikzpicture}[scale=0.8,rotate=90]
\tikzstyle{vertex}=[draw,circle,font=\scriptsize,minimum size=6pt,inner sep=1pt,fill=blue!40!white]
\tikzstyle{edge}=[draw,thick]
\node[vertex] (v1) at (0, 0) {};
\node[vertex] (v2) at (3, 0) {};
\node[vertex] (v6) at (0, -4) {};
\node[vertex] (v5) at (3, -4) {};
\node[vertex] (v3) at (3, -1.3) {};
\node[vertex] (v4) at (3, -2.6) {};
\node[vertex] (v7) at (0, -2.3) {};
\node[vertex] (v8) at (0, -1.1) {};
\node[vertex] (v9) at (1.2, -3.3) {};
\node[vertex] (v10) at (1.2, -2.3) {};
\node[vertex] (v11) at (1.0, -1.3) {};
\node[vertex] (v12) at (0.6, -2.5) {};
\node[vertex] (v13) at (2.5, -2.0) {};
\node[vertex] (v14) at (2.2, -2.7) {};
\node[vertex] (v15) at (1.8, -1.9) {};
\node[vertex] (v16) at (1.6, -1.1) {};
\node[vertex] (v19) at (2.6, -0.4) {};
\node[vertex] (v17) at (1.9, -0.5) {};
\node[vertex] (v18) at (2.2, -1.3) {};
\path[edge] (v1) -- (v2) -- (v3) -- (v4) -- (v5) -- (v6) -- (v7) -- (v8) -- (v1);
\path[edge] (v6) -- (v9) -- (v10) -- (v11) -- (v8);
\path[edge] (v7) -- (v12) -- (v10);
\path[edge] (v4) -- (v13) -- (v3);
\path[edge] (v5) -- (v14) -- (v13);
\path[edge] (v14) -- (v15) -- (v16) -- (v11);
\path[edge] (v16) -- (v17) -- (v19) -- (v18) -- (v15);
\path[edge] (v19) -- (v2);
\end{tikzpicture} }

\vspace{\baselineskip}
\subcaptionbox{\label{subfig:6h}$P_1^{(20)}$}
{ \begin{tikzpicture}[scale=0.8,rotate=90]
\tikzstyle{vertex}=[draw,circle,font=\scriptsize,minimum size=6pt,inner sep=1pt,fill=blue!40!white]
\tikzstyle{edge}=[draw,thick]
\node[vertex] (v1) at (0, 0) {};
\node[vertex] (v3) at (3, 0) {};
\node[vertex] (v8) at (0, -4) {};
\node[vertex] (v6) at (3, -4) {};
\node[vertex] (v2) at (1.5, 0) {};
\node[vertex] (v4) at (3, -1.3) {};
\node[vertex] (v5) at (3, -2.7) {};
\node[vertex] (v7) at (1.5, -4) {};
\node[vertex] (v13) at (0, -0.6) {};
\node[vertex] (v12) at (0, -1.3) {};
\node[vertex] (v11) at (0, -2.0) {};
\node[vertex] (v10) at (0, -2.6) {};
\node[vertex] (v9) at (0, -3.3) {};
\node[vertex] (v14) at (0.8, -3.4) {};
\node[vertex] (v15) at (1.8, -3.2) {};
\node[vertex] (v16) at (1.6, -2.5) {};
\node[vertex] (v17) at (1.2, -2.0) {};
\node[vertex] (v18) at (2.2, -1.0) {};
\node[vertex] (v19) at (0.5, -1.3) {};
\node[vertex] (v20) at (0.5, -0.7) {};
\node[vertex] (v21) at (1.1, -1.0) {};
\node[vertex] (v22) at (1.6, -0.7) {};
\node[vertex] (v23) at (1.2, -0.4) {};
\path[edge] (v21) -- (v22) -- (v23) -- (v2);
\path[edge] (v12) -- (v19) -- (v20) -- (v13);
\path[edge] (v20) -- (v21) -- (v19);
\path[edge] (v1) -- (v2) -- (v3) -- (v4) -- (v5) -- (v6) -- (v7) -- (v8) -- (v9) -- (v10)  -- (v11)  -- (v12) -- (v13) -- (v1);
\path[edge] (v9) -- (v14) -- (v7);
\path[edge] (v14) -- (v15) -- (v6);
\path[edge] (v15) -- (v16) -- (v17) -- (v11);
\path[edge] (v10) -- (v16);
\path[edge] (v17) -- (v18);
\path[edge] (v3) -- (v18) -- (v5);
\end{tikzpicture} }
\subcaptionbox{\label{subfig:6i}$P_2^{(20)}$}
{ \begin{tikzpicture}[scale=0.8,rotate=90]
\tikzstyle{vertex}=[draw,circle,font=\scriptsize,minimum size=6pt,inner sep=1pt,fill=blue!40!white]
\tikzstyle{edge}=[draw,thick]
\node[vertex] (v1) at (0, 0) {};
\node[vertex] (v2) at (3, 0) {};
\node[vertex] (v8) at (0, -4) {};
\node[vertex] (v5) at (3, -4) {};
\node[vertex] (v7) at (1, -4) {};
\node[vertex] (v6) at (2, -4) {};
\node[vertex] (v3) at (3, -1.3) {};
\node[vertex] (v4) at (3, -2.7) {};
\node[vertex] (v12) at (0, -0.8) {};
\node[vertex] (v11) at (0, -1.6) {};
\node[vertex] (v10) at (0, -2.4) {};
\node[vertex] (v9) at (0, -3.2) {};
\node[vertex] (v14) at (0.6, -0.9) {};
\node[vertex] (v13) at (0.6, -1.5) {};
\node[vertex] (v15) at (1.3, -1.2) {};
\node[vertex] (v16) at (1.8, -0.8) {};
\node[vertex] (v17) at (2.4, -0.6) {};
\node[vertex] (v19) at (2.4, -2.2) {};
\node[vertex] (v18) at (2.5, -1.6) {};
\node[vertex] (v20) at (1.6, -2.3) {};
\node[vertex] (v21) at (1.4, -2.8) {};
\node[vertex] (v22) at (1.2, -3.4) {};
\node[vertex] (v23) at (2.4, -3.4) {};
\path[edge] (v8) -- (v22) -- (v6);
\path[edge] (v22) -- (v23) -- (v5);
\path[edge] (v23) -- (v21) -- (v9);
\path[edge] (v10) -- (v20) -- (v19);
\path[edge] (v21) -- (v20);
\path[edge] (v3) -- (v18) -- (v19) -- (v4);
\path[edge] (v11) -- (v13) -- (v14) -- (v12);
\path[edge] (v13) -- (v15) -- (v14);
\path[edge] (v15) -- (v16) -- (v17) -- (v2);
\path[edge] (v1) -- (v2) -- (v3) -- (v4) -- (v5) -- (v6) -- (v7) -- (v8) -- (v9) -- (v10)  -- (v11)  -- (v12) -- (v1);
\end{tikzpicture} }
\subcaptionbox{\label{subfig:6j}$P_3^{(20)}$}
{ \begin{tikzpicture}[scale=0.8,rotate=90]
\tikzstyle{vertex}=[draw,circle,font=\scriptsize,minimum size=6pt,inner sep=1pt,fill=blue!40!white]
\tikzstyle{edge}=[draw,thick]
\node[vertex] (v1) at (0, 0) {};
\node[vertex] (v3) at (3, 0) {};
\node[vertex] (v7) at (0, -4) {};
\node[vertex] (v6) at (3, -4) {};
\node[vertex] (v2) at (2, 0) {};
\node[vertex] (v4) at (3, -1.2) {};
\node[vertex] (v5) at (3, -3.2) {};
\node[vertex] (v11) at (0, -0.8) {};
\node[vertex] (v10) at (0, -1.6) {};
\node[vertex] (v9) at (0, -2.4) {};
\node[vertex] (v8) at (0, -3.2) {};
\node[vertex] (v12) at (0.8, -0.5) {};
\node[vertex] (v13) at (1, -3.5) {};
\node[vertex] (v14) at (1.8, -3.6) {};
\node[vertex] (v15) at (1.9, -3.0) {};
\node[vertex] (v16) at (2.4, -2.6) {};
\node[vertex] (v17) at (0.8, -2.8) {};
\node[vertex] (v18) at (0.7, -2.0) {};
\node[vertex] (v19) at (1.6, -2.4) {};
\node[vertex] (v20) at (1.2, -1.4) {};
\node[vertex] (v21) at (1.8, -1.7) {};
\node[vertex] (v22) at (1.6, -0.8) {};
\node[vertex] (v23) at (2.3, -1.1) {};
\path[edge] (v20) -- (v22) -- (v23) -- (v21);
\path[edge] (v17) -- (v19) -- (v21) -- (v20) -- (v18);
\path[edge] (v8) -- (v17) -- (v18) -- (v9);
\path[edge] (v7) -- (v13) -- (v14) -- (v6);
\path[edge] (v14) -- (v15) -- (v16) -- (v5);
\path[edge] (v10) -- (v12) -- (v2);
\path[edge] (v1) -- (v12);
\path[edge] (v23) -- (v4);
\path[edge] (v16) -- (v19);
\path[edge] (v1) -- (v2) -- (v3) -- (v4) -- (v5) -- (v6) -- (v7) -- (v8) -- (v9) -- (v10)  -- (v11)  -- (v1);
\end{tikzpicture} }
\subcaptionbox{\label{subfig:6k}$P_4^{(20)}$}
{ \begin{tikzpicture}[scale=0.8,rotate=90]
\tikzstyle{vertex}=[draw,circle,font=\scriptsize,minimum size=6pt,inner sep=1pt,fill=blue!40!white]
\tikzstyle{edge}=[draw,thick]
\node[vertex] (v1) at (0, 0) {};
\node[vertex] (v4) at (3, 0) {};
\node[vertex] (v7) at (0, -4) {};
\node[vertex] (v6) at (3, -4) {};
\node[vertex] (v2) at (1, 0) {};
\node[vertex] (v3) at (2, 0) {};
\node[vertex] (v5) at (3, -3) {};
\node[vertex] (v8) at (0, -3.2) {};
\node[vertex] (v9) at (0, -1.8) {};
\node[vertex] (v10) at (0.6, -2.8) {};
\node[vertex] (v11) at (0.6, -2.2) {};
\node[vertex] (v11) at (0.6, -2.2) {};
\node[vertex] (v12) at (0.8, -1.5) {};
\node[vertex] (v13) at (1.6, -3.2) {};
\node[vertex] (v14) at (2.7, -1.5) {};
\node[vertex] (v15) at (1.2, -1.0) {};
\node[vertex] (v16) at (1.2, -0.5) {};
\path[edge] (v12) -- (v15) -- (v16) -- (v2);
\node[vertex] (v17) at (2.3, -2.1) {};
\node[vertex] (v18) at (1.5, -1.8) {};
\node[vertex] (v19) at (2.1, -0.4) {};
\node[vertex] (v20) at (1.9, -0.8) {};
\node[vertex] (v21) at (1.7, -1.3) {};
\node[vertex] (v22) at (2.1, -1.4) {};
\node[vertex] (v23) at (2.4, -0.9) {};
\path[edge] (v18) -- (v21) -- (v20) -- (v19) -- (v3);
\path[edge] (v21) -- (v22) -- (v23) -- (v19);
\path[edge] (v17) -- (v23);
\path[edge] (v15) -- (v18) -- (v17) -- (v14);
\path[edge] (v5) -- (v14) -- (v4);
\path[edge] (v1) -- (v12) -- (v13);
\path[edge] (v7) -- (v13) -- (v6);
\path[edge] (v8) -- (v10) -- (v11) -- (v9);
\path[edge] (v1) -- (v2) -- (v3) -- (v4) -- (v5) -- (v6) -- (v7) -- (v8) -- (v9) -- (v1);
\end{tikzpicture} }

\vspace{\baselineskip}
\subcaptionbox{\label{subfig:6l}$P_1^{(24)}$}
{ \begin{tikzpicture}[scale=0.8,rotate=90]
\tikzstyle{vertex}=[draw,circle,font=\scriptsize,minimum size=6pt,inner sep=1pt,fill=blue!40!white]
\tikzstyle{edge}=[draw,thick]
\node[vertex] (v1) at (0, 0) {};
\node[vertex] (v2) at (3, 0) {};
\node[vertex] (v8) at (0, -4) {};
\node[vertex] (v7) at (3, -4) {};
\node[vertex] (v3) at (3, -1.2) {};
\node[vertex] (v4) at (3, -2) {};
\node[vertex] (v5) at (3, -2.8) {};
\node[vertex] (v6) at (3, -3.4) {};
\node[vertex] (v11) at (0, -0.8) {};
\node[vertex] (v10) at (0, -2.4) {};
\node[vertex] (v9) at (0, -3.2) {};
\node[vertex] (v12) at (0.8, -3.5) {};
\node[vertex] (v13) at (1.4, -3.2) {};
\node[vertex] (v14) at (2.2, -3.5) {};
\node[vertex] (v15) at (0.7, -2.9) {};
\node[vertex] (v16) at (1.5, -2.3) {};
\node[vertex] (v17) at (2.2, -1.6) {};
\node[vertex] (v18) at (0.5, -2.0) {};
\node[vertex] (v19) at (1.2, -1.8) {};
\node[vertex] (v20) at (1.8, -1.0) {};
\node[vertex] (v21) at (2.4, -0.6) {};
\node[vertex] (v22) at (0.5, -1.4) {};
\node[vertex] (v23) at (1.0, -1.4) {};
\node[vertex] (v24) at (1.6, -0.5) {};
\node[vertex] (v25) at (0.5, -0.7) {};
\node[vertex] (v26) at (0.5, -0.3) {};
\node[vertex] (v27) at (1.1, -0.5) {};
\path[edge] (v11) -- (v25) -- (v27) -- (v26) -- (v1);
\path[edge] (v24) -- (v27);
\path[edge] (v23) -- (v22) -- (v24);
\path[edge] (v10) -- (v18) -- (v19) -- (v20) -- (v21) -- (v2);
\path[edge] (v9) -- (v15) -- (v16) -- (v4);
\path[edge] (v16) -- (v17) -- (v3);
\path[edge] (v21) -- (v17);
\path[edge] (v18) -- (v22);
\path[edge] (v19) -- (v23);
\path[edge] (v20) -- (v24);
\path[edge] (v12) -- (v15);
\path[edge] (v13) -- (v5);
\path[edge] (v8) -- (v12) -- (v13) -- (v14) -- (v6);
\path[edge] (v1) -- (v2) -- (v3) -- (v4) -- (v5) -- (v6) -- (v7) -- (v8) -- (v9)  -- (v10)  -- (v11) -- (v1);
\end{tikzpicture} }
\subcaptionbox{\label{subfig:6m}$P_2^{(24)}$}
{ \begin{tikzpicture}[scale=0.8,rotate=90]
\tikzstyle{vertex}=[draw,circle,font=\scriptsize,minimum size=6pt,inner sep=1pt,fill=blue!40!white]
\tikzstyle{edge}=[draw,thick]
\node[vertex] (v1) at (0, 0) {};
\node[vertex] (v2) at (3, 0) {};
\node[vertex] (v6) at (0, -4) {};
\node[vertex] (v3) at (3, -4) {};
\node[vertex] (v4) at (2.3, -4) {};
\node[vertex] (v5) at (1.5, -4) {};
\node[vertex] (v7) at (0, -3.1) {};
\node[vertex] (v8) at (0, -1.6) {};
\node[vertex] (v9) at (0.5, -3.5) {};
\node[vertex] (v10) at (1, -3.5) {};
\node[vertex] (v11) at (1.9, -3.5) {};
\node[vertex] (v12) at (2.1, -3.1) {};
\node[vertex] (v13) at (2.6, -3.2) {};
\node[vertex] (v14) at (0.5, -3.0) {};
\node[vertex] (v15) at (0.8, -2.5) {};
\node[vertex] (v16) at (1.6, -2.5) {};
\node[vertex] (v17) at (1.9, -2.0) {};
\node[vertex] (v18) at (2.3, -1.8) {};
\node[vertex] (v19) at (2.7, -2.1) {};
\node[vertex] (v20) at (2.2, -1.1) {};
\node[vertex] (v21) at (2.6, -0.6) {};
\node[vertex] (v22) at (0.4, -2.0) {};
\node[vertex] (v23) at (0.8, -1.6) {};
\node[vertex] (v24) at (1.2, -2.0) {};
\node[vertex] (v25) at (0.4, -1.2) {};
\node[vertex] (v26) at (0.7, -0.5) {};
\node[vertex] (v27) at (1.4, -1.4) {};
\path[edge] (v26) -- (v27) -- (v16);
\path[edge] (v22) -- (v25);
\path[edge] (v24) -- (v27);
\path[edge] (v8) -- (v25) -- (v26) -- (v1);
\path[edge] (v15) -- (v22) -- (v23) -- (v24) -- (v15);
\path[edge] (v21) -- (v2);
\path[edge] (v18) -- (v20) -- (v21) -- (v19);
\path[edge] (v12) -- (v16) -- (v17) -- (v18) -- (v19) -- (v13);
\path[edge] (v7) -- (v14) -- (v9);
\path[edge] (v15) -- (v14);
\path[edge] (v4) -- (v11) -- (v12) -- (v13) -- (v3);
\path[edge] (v6) -- (v9) -- (v10) -- (v5);
\path[edge] (v1) -- (v2) -- (v3) -- (v4) -- (v5) -- (v6) -- (v7) -- (v8) -- (v1);
\end{tikzpicture} }
\subcaptionbox{\label{subfig:6n}$P_3^{(24)}$}
{ \begin{tikzpicture}[scale=0.8,rotate=90]
\tikzstyle{vertex}=[draw,circle,font=\scriptsize,minimum size=6pt,inner sep=1pt,fill=blue!40!white]
\tikzstyle{edge}=[draw,thick]
\node[vertex] (v1) at (0, 0) {};
\node[vertex] (v2) at (3, 0) {};
\node[vertex] (v8) at (0, -4) {};
\node[vertex] (v7) at (3, -4) {};
\node[vertex] (v3) at (3, -1.4) {};
\node[vertex] (v4) at (3, -2.2) {};
\node[vertex] (v5) at (3, -2.8) {};
\node[vertex] (v6) at (3, -3.4) {};
\node[vertex] (v10) at (0, -1.7) {};
\node[vertex] (v9) at (0, -2.4) {};
\node[vertex] (v11) at (0.7, -3.5) {};
\node[vertex] (v12) at (1.4, -3.4) {};
\node[vertex] (v13) at (2.1, -3.6) {};
\node[vertex] (v14) at (0.7, -2.8) {};
\node[vertex] (v15) at (1.4, -2.7) {};
\node[vertex] (v16) at (1.9, -2.2) {};
\node[vertex] (v17) at (2.4, -1.0) {};
\node[vertex] (v18) at (0.5, -2.2) {};
\node[vertex] (v19) at (1.0, -1.9) {};
\node[vertex] (v20) at (1.7, -1.6) {};
\node[vertex] (v21) at (0.4, -1.6) {};
\node[vertex] (v22) at (1.7, -1.1) {};
\node[vertex] (v23) at (0.6, -1.1) {};
\node[vertex] (v24) at (1.1, -1.0) {};
\node[vertex] (v25) at (1.3, -0.6) {};
\node[vertex] (v27) at (0.7, -0.3) {};
\node[vertex] (v26) at (0.5, -0.6) {};
\path[edge] (v27) -- (v26)-- (v24);
\path[edge] (v1) -- (v27)-- (v25);
\path[edge] (v21) -- (v23)-- (v24) -- (v25) -- (v22);
\path[edge] (v18) -- (v21)-- (v22) -- (v20);
\path[edge] (v10) -- (v18) -- (v19)-- (v20) -- (v17);
\path[edge] (v16) -- (v17)-- (v2);
\path[edge] (v15) -- (v16)-- (v3);
\path[edge] (v14) -- (v15)-- (v4);
\path[edge] (v12) -- (v5);
\path[edge] (v11) -- (v14) -- (v9);
\path[edge] (v8) -- (v11) -- (v12) -- (v13) -- (v6);
\path[edge] (v12) -- (v5);
\path[edge] (v1) -- (v2) -- (v3) -- (v4) -- (v5) -- (v6) -- (v7) -- (v8) -- (v9)  -- (v10) -- (v1);
\end{tikzpicture} }
\subcaptionbox{\label{subfig:6o}$P_4^{(24)}$}
{ \begin{tikzpicture}[scale=0.8,rotate=90]
\tikzstyle{vertex}=[draw,circle,font=\scriptsize,minimum size=6pt,inner sep=1pt,fill=blue!40!white]
\tikzstyle{edge}=[draw,thick]
\node[vertex] (v1) at (0, 0) {};
\node[vertex] (v4) at (3, 0) {};
\node[vertex] (v8) at (0, -4) {};
\node[vertex] (v6) at (3, -4) {};
\node[vertex] (v2) at (1.5, 0) {};
\node[vertex] (v3) at (2.2, 0) {};
\node[vertex] (v5) at (3, -2.5) {};
\node[vertex] (v7) at (1.9, -4) {};
\node[vertex] (v11) at (0, -1) {};
\node[vertex] (v10) at (0, -2.4) {};
\node[vertex] (v9) at (0, -3.2) {};
\node[vertex] (v12) at (2.6, -3.5) {};
\node[vertex] (v13) at (2.5, -3.0) {};
\node[vertex] (v14) at (1.7, -3.5) {};
\node[vertex] (v15) at (2.0, -3.1) {};
\node[vertex] (v16) at (1.3, -2.8) {};
\node[vertex] (v17) at (1.8, -2.5) {};
\node[vertex] (v18) at (2.4, -2.2) {};
\node[vertex] (v19) at (0.8, -3.0) {};
\node[vertex] (v20) at (0.4, -3.4) {};
\node[vertex] (v21) at (1.1, -2.2) {};
\node[vertex] (v22) at (1.7, -1.8) {};
\node[vertex] (v23) at (2.1, -1.3) {};
\node[vertex] (v24) at (0.3, -1.8) {};
\node[vertex] (v25) at (0.8, -1.7) {};
\node[vertex] (v26) at (0.7, -1.0) {};
\node[vertex] (v27) at (1.0, -0.4) {};
\path[edge] (v26) -- (v27) -- (v1);
\path[edge] (v25) -- (v26) -- (v11);
\path[edge] (v10) -- (v24) -- (v25) -- (v21);
\path[edge] (v22) -- (v2);
\path[edge] (v23) -- (v3);
\path[edge] (v19) -- (v21) -- (v22) -- (v23) -- (v18);
\path[edge] (v9) -- (v20) -- (v19) -- (v16);
\path[edge] (v17) -- (v15);
\path[edge] (v14) -- (v16) -- (v17) -- (v18) -- (v4);
\path[edge] (v7) -- (v14) -- (v15) -- (v13);
\path[edge] (v6) -- (v12) -- (v13) -- (v5);
\path[edge] (v1) -- (v2) -- (v3) -- (v4) -- (v5) -- (v6) -- (v7) -- (v8) -- (v9)  -- (v10)  -- (v11) -- (v1);
\end{tikzpicture} }
\caption{The proto-gadget graphs for the proof of Theorem~\ref{thm:2} for cases $d \in \{12, 16, 20, 24\}$.
The set $P_i^{(d)}$ is used to construct the decorating gadgets $Q_i$, as described in the proof.}
\label{fig:protoGadgetsExtra}
\end{figure}
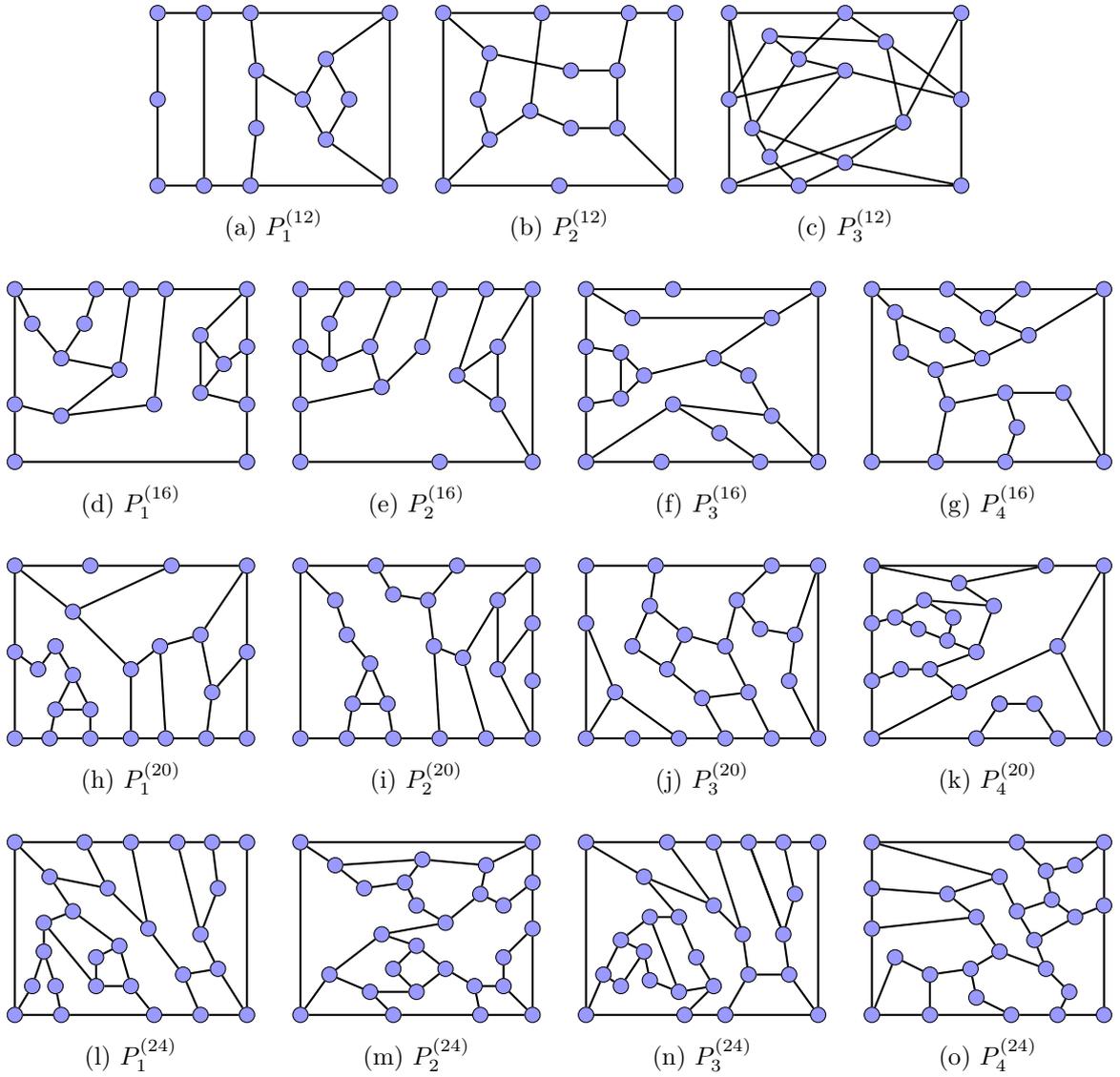
\end{proof}

We note that there are other strategies for the choice of gadgets in the proof.
For example, one may prefer to find one gadget and then, to generate the others, repeatedly apply a construction that preserves
symmetry but does not produce vertices of unwanted degree. One candidate is the so-called Fowler construction \cite{basic2023vertex}.
This approach would lead to a nut graph of yet larger order than the one generated by the present proof.
The order of graph $G$ constructed in the proof of Theorem~\ref{thm:2} is $\omega(d)|V(H)|$, where $\omega(8) = 53$, $\omega(12) = 99$,
$\omega(16) = 161$, $\omega(20) = 241$, and $\omega(24) = 337$. Recall that $H$ denotes a $(d/2)$-regular starting graph that realises $\mathfrak{G}$.

\section{Discussion}

Constructive methods used to answer K{\H o}nig's question typically do not provide minimal examples.
Let $\alpha(\mathfrak{G})$ be the smallest order of the graphs representing the group $\mathfrak{G}$.
Sabidussi~\cite{Sabidussi1959} opened the question by studying the order of $\alpha(\mathfrak{G})$ with respect to $|\mathfrak{G}|$.
The value $\alpha(\mathfrak{G})$ has been determined for various families of groups (see \cite{Arlinghaus1985} for abelian groups and a survey in \cite{Woodruff2016}). 
Babai \cite{Babai1974} gave
$\alpha(\mathfrak{G}) \leq 2| \mathfrak{G}|$ provided that $\mathfrak{G} \notin \{ \mathbb{Z}_3, \mathbb{Z}_4, \mathbb{Z}_5  \}$;
 Deligeorgaki \cite{Deligeorgaki2022} improved this to
$\alpha(\mathfrak{G}) \leq | \mathfrak{G}|$,
with a longer list of exceptions that includes some infinite families.
Planar graphs have also been considered from this point of view (for a survey see \cite{Jones2021}).

Nut graphs raise analogous questions. It is clear that the constructions 
using in proving Theorems~\ref{thm:1} and~\ref{thm:2} are far from minimal.
As an example, consider the group $\mathfrak{G}_{288}$ of order $288$,
defined by its permutation representation
$$\mathfrak{G}_{288} = \langle 
(1,2,3)(4,5)(6,7,8),\ (1,8)(2,7)(3,6)(4,9)(5,10),\ (7,8)
\rangle.$$
In GAP~\cite{GAP4}, this group can be obtained by calling \verb+SmallGroup(288, 889)+.
The smallest $4$-regular graph representing this group that is given by 
Sabidussi's construction (Theorem~\ref{thm:sabidussi}) is of order  $5760$.
Expansion to a nut graph by the construction used in the proof of Theorem~\ref{thm:1}
gives order $109440$.
A much smaller $4$-regular parent graph could have been used as the basis for 
that construction, since the smallest $4$-regular graph representing 
$\mathfrak{G}_{288}$ is of order $11$; see Figure~\ref{fig:7}(a), leading to a nut graph of order $209$.
However,  the database obtained by \verb+nutgen+ \cite{nutgen-site} reveals that 
the smallest nut graph that represents $\mathfrak{G}_{288}$ has only $10$ vertices; see Figure~\ref{fig:7}(b).

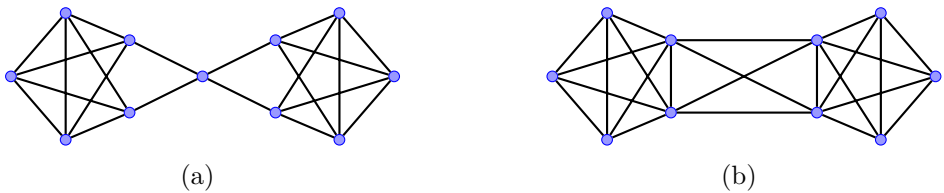
\begin{figure}[!htbp]
\centering
\subcaptionbox{\label{subfig:7a}} {
\begin{tikzpicture}[scale=1.2]
\tikzstyle{vertex}=[draw,circle,font=\scriptsize,minimum size=4pt,inner sep=1pt,color=blue,fill=blue!40!white]
\tikzstyle{edge}=[draw,thick]
\node[vertex] (v1) at (0.2, 0) {};
\node[vertex] (v2) at (0.2, 0.8) {};
\node[vertex] (v3) at (-0.5, -0.3) {};
\node[vertex] (v4) at (-0.5, 1.1) {};
\node[vertex] (v5) at (-1.1, 0.4) {};
\path[edge]  (v2) -- (v3) -- (v4) -- (v5) -- (v1) -- (v3) -- (v5) -- (v2) -- (v4) -- (v1);
\node[vertex] (av1) at (2-0.2, 0) {};
\node[vertex] (av2) at (2-0.2, 0.8) {};
\node[vertex] (av3) at (2+0.5, -0.3) {};
\node[vertex] (av4) at (2+0.5, 1.1) {};
\node[vertex] (av5) at (2+1.1, 0.4) {};
\node[vertex] (central) at (1, 0.4) {};
\path[edge] (av1) -- (central);
\path[edge] (v1) -- (central);
\path[edge] (av2) -- (central);
\path[edge] (v2) -- (central);
\path[edge] (av2) -- (av3) -- (av4) -- (av5) -- (av1) -- (av3) -- (av5) -- (av2) -- (av4) -- (av1);
 \end{tikzpicture} }
\qquad\qquad
\subcaptionbox{\label{subfig:7b}} {
\begin{tikzpicture}[scale=1.2]
\tikzstyle{vertex}=[draw,circle,font=\scriptsize,minimum size=4pt,inner sep=1pt,color=blue,fill=blue!40!white]
\tikzstyle{edge}=[draw,thick]
\node[vertex] (v1) at (0.2, 0) {};
\node[vertex] (v2) at (0.2, 0.8) {};
\node[vertex] (v3) at (-0.5, -0.3) {};
\node[vertex] (v4) at (-0.5, 1.1) {};
\node[vertex] (v5) at (-1.1, 0.4) {};
\path[edge] (v1) -- (v2) -- (v3) -- (v4) -- (v5) -- (v1) -- (v3) -- (v5) -- (v2) -- (v4) -- (v1);
\node[vertex] (av1) at (2-0.2, 0) {};
\node[vertex] (av2) at (2-0.2, 0.8) {};
\node[vertex] (av3) at (2+0.5, -0.3) {};
\node[vertex] (av4) at (2+0.5, 1.1) {};
\node[vertex] (av5) at (2+1.1, 0.4) {};
\path[edge] (av1) -- (v1) -- (av2) -- (v2) -- (av1);
\path[edge] (av1) -- (av2) -- (av3) -- (av4) -- (av5) -- (av1) -- (av3) -- (av5) -- (av2) -- (av4) -- (av1);
 \end{tikzpicture} }
\caption{(a) The smallest $4$-regular graph, and (b) the smallest nut graph, that represent $\mathfrak{G}_{288}$.}
\label{fig:7}
\end{figure}

Let $\beta(\mathfrak{G})$ be the smallest order of the nut graphs representing the group $\mathfrak{G}$.
It is evident that $\beta(\mathfrak{G}) \geq \alpha(\mathfrak{G})$. For groups up to order $6$, the values are
\begin{align*}
\alpha(\mathbb{Z}_1) & = 1, \beta(\mathbb{Z}_1) = 9; & 
\alpha(\mathbb{Z}_2) & = 2, \beta(\mathbb{Z}_2) = 8; & 
\alpha(\mathbb{Z}_3) & = 9, \beta(\mathbb{Z}_3) = 11; \\
\alpha(\mathbb{Z}_4) & = 10, \beta(\mathbb{Z}_4) = 11; &
\alpha(\mathbb{Z}_2 \times \mathbb{Z}_2) & = 4, \beta(\mathbb{Z}_2 \times \mathbb{Z}_2 ) = 7; &
\alpha(\mathbb{Z}_5) & = 15, \beta(\mathbb{Z}_5) = 15; \\ 
\alpha(\mathbb{Z}_3 \times \mathbb{Z}_2 ) & = 11, \beta(\mathbb{Z}_3 \times \mathbb{Z}_2 ) = 11; &
\alpha(\mathbb{Z}_3 \rtimes \mathbb{Z}_2 ) &= 3, \beta(\mathbb{Z}_3 \rtimes \mathbb{Z}_2 ) = 7.
\end{align*}
These numbers were found by computer search of the available censuses of nut graphs \cite{hog,HoG2}. For $\mathbb{Z}_5$,
\cite[Lemma 5.2]{Arlinghaus1985} gives us $\alpha(\mathbb{Z}_5) = 15 \leq \beta(\mathbb{Z}_5)$. The equality $\beta(\mathbb{Z}_5) = 15$
was established by finding an example.

\begin{problem}
\label{problem:2}
Given any finite group $\mathfrak{G}$, find a nut graph $G$ of minimum order, such that $\Aut(G) \cong \mathfrak{G}$.
 Find an upper bound on $\beta(\mathfrak{G})$ in terms of $|\mathfrak{G}|$.
\end{problem}

Another question relates to the degrees of regular nut graphs that represent groups $\mathfrak{G}$.

\begin{problem}
\label{problem:1}
Given a finite group $\mathfrak{G}$ and an integer $d \geq 3$, find a $d$-regular nut graph $G$, such that $\Aut(G) \cong \mathfrak{G}$.
\end{problem}

Theorem~\ref{thm:2} supplies the answer for 
small cases $d \equiv 0 \pmod{4}$, and the same technique could be used to extend the list, 
but this still leaves unresolved all cases with $d \not\equiv 0 \pmod{4}$.
A missing case of particular interest is $d =3$.
The graphs that can be used to model conjugated carbon frameworks in H\"{u}ckel theory and similar applications~\cite{Streitwieser1961}
are known as chemical graphs. A \emph{chemical graph} in this definition is connected and subcubic.
Cubic chemical graphs form an important subclass that includes the fullerenes \cite{fowler2007atlas}. Applications of 
nut graphs in theories of radical chemistry and molecular conduction are described in \cite{nuttybook}.
Interestingly, the eponymous Frucht graph, introduced in \cite{Frucht1949}
as a small graph that has trivial symmetry, is both cubic (in fact polyhedral) and a nut graph.
It has order $12$ and is the smallest cubic nut graph of trivial symmetry.
 Frucht also threated the group $\mathbb{Z}_2$ separately;
his graph that realises this group (see \cite[Fig.~2]{Frucht1949}) is not a nut graph.
It is straightforward to show that his general constructions~\cite{Frucht1949} for groups of order greater than $2$
do not yield nut graphs. The repeated motifs devised by Frucht (the `corners' \cite{Frucht1949}) give rise to at least one non-trivial kernel eigenvector with 
some zero entries in the constructed graph.
Hence, it would be interesting to find a construction that yields cubic nut graphs directly.

\section*{Acknowledgements}
We would like to thank our colleague Prof.~Primož Potočnik for fruitful discussion and for drawing our attention
to \cite{Frucht1949} during a research visit to Sheffield.
The work of Nino Bašić is supported in part by the Slovenian Research Agency (research program P1-0294
and research projects N1-0140 and J1-2481).
PWF thanks the Leverhulme Trust for an
Emeritus Fellowship on the theme of
{\lq Modelling molecular currents, conduction and aromaticity\rq} and the Francqui Foundation for the award of
an International Francqui Professorship.

\nocite{GitHub}

\bibliographystyle{amcjoucc}
\bibliography{references}

\end{document}